\documentclass[11pt,reqno,tbtags]{amsart}
\usepackage{amssymb}
\usepackage{hyperref}
\usepackage{upref}
\usepackage{natbib}
\bibpunct[, ]{[}{]}{;}{n}{,}{,}

\numberwithin{equation}{section}

\allowdisplaybreaks

\newtheorem{theorem}{Theorem}[section]
\newtheorem{lemma}[theorem]{Lemma}

\theoremstyle{definition}

\newtheorem*{definition}{Definition}

\newtheorem{remark}[theorem]{Remark}

\theoremstyle{remark}

\newenvironment{romenumerate}[1][0pt]{% optional argument changes indentation
\addtolength{\leftmargini}{#1}\begin{enumerate}% gives (i), (ii) etc.
 }{\end{enumerate}}

\newcounter{oldenumi}
{\setcounter{oldenumi}{\value{enumi}}
\begin{romenumerate} \setcounter{enumi}{\value{oldenumi}}}
{\end{romenumerate}}

\newcounter{thmenumerate}

\newcounter{romxenumerate}   %less indented than standard.

\newcounter{xenumerate}   %no left indentation; thus wider lines

\newcommand{\refT}[1]{Theorem~\ref{#1}}

\newcommand{\refL}[1]{Lemma~\ref{#1}}

\newcommand{\refS}[1]{Section~\ref{#1}}

\newcommand\marginal[1]{\marginpar{\raggedright\parindent=0pt\tiny #1}}

\begingroup
  \divide\count255 by 60
  \multiply\count255 by -60
  \advance\count255 by \time
  \ifnum \count255 < 10 \xdef\klockan{\the\count1.0\the\count255}
  \else\xdef\klockan{\the\count1.\the\count255}\fi
\endgroup

\newcommand\set[1]{\ensuremath{\{#1\}}}

\newcommand\xpar[1]{(#1)}
\newcommand\bigpar[1]{\bigl(#1\bigr)}

\newcommand\lrpar[1]{\left(#1\right)}
\newcommand\bigsqpar[1]{\bigl[#1\bigr]}

\newcommand\lrsqpar[1]{\left[#1\right]}

\def\rompar(#1){\textup(#1\textup)}    % usage: \rompar(...)

\def\xexp(#1){e^{#1}}

\newcommand\ntoo{\ensuremath{{n\to\infty}}}

\newcommand\eg{e.g.\spacefactor=1000}

\newcommand{\tend}{\longrightarrow}
\newcommand\dto{\overset{\mathrm{d}}{\tend}}

\newcounter{CC}
\newcounter{cc}
\newcommand{\cc}{\stepcounter{cc}\ccx} %new constant c_i
\newcommand{\ccx}{c_{\arabic{cc}}}     %repeats the last c_i
\newcommand\E{\operatorname{\mathbb E{}}}

\newcommand\Var{\operatorname{Var}}

\newcommand\ga{\alpha}
\newcommand\gb{\beta}

\newcommand\gG{\Gamma}

\newcommand\kk{\varkappa}

\newcommand\gs{\sigma}
\newcommand\gss{\sigma^2}

\newcommand\cA{\mathcal A}

\newcommand\cC{\mathcal C}

\def\[#1]{[\![#1]\!]}

\newcommand\setoi{\set{0,1}}

\newcommand\SW{W}
\newcommand\REM[1]{{\raggedright\texttt{[#1]}\par\marginal{XXX}}}

\newenvironment{comment}{\setbox0=\vbox\bgroup}{\egroup} %deletes!

\newcommand\urladdrx[1]{{\urladdr{\def~{\Tilde}#1}}}

\renewcommand\Tilde{{\tiny$\sim$}}

\begin{document}

%\title[On the Statistics of the Number of Fixed-Dimensional Subcubes in a Random Subset of the n-Dimensional Discrete Unit Cube]
\title[Statistics of the Number of Subcubes in Subsets of the n-Cube ]
{
On the Statistics of the Number of Fixed-Dimensional Subcubes in a Random Subset of the n-Dimensional Discrete Unit Cube
}
\thanks{Accompanied by Maple package SMCboole.txt, available from
\hfill\break
\href{http://www.math.rutgers.edu/~zeilberg/mamarim/mamarimhtml/subcubes.html}{\texttt{
http://www.math.rutgers.edu/\Tilde zeilberg/mamarim/mamarimhtml/subcubes.html}}.\\
}

\date{Feb. 17, 2023} % (typeset \today{} \klockan)} %; revised ...

\author{Svante Janson}
\address{Department of Mathematics, Uppsala University, PO Box 480,
SE-751~06 Uppsala, Sweden}
\email{svante.janson [At] math [Dot] uu [Dot] se}
\urladdrx{\href{http://www.math.uu.se/~svante/}{http://www.math.uu.se/~svante/}}

\author{Blair Seidler}
\address{ Department of Mathematics,
Rutgers University (New Brunswick), Piscataway NJ 08854, USA}
\email{blair [At] math [Dot] rutgers [Dot] edu}

\author{Doron ZEILBERGER}
\address{Department of Mathematics, Rutgers University
(New Brunswick), Piscataway, NJ 08854, USA}
\email{zeilberg [At] math [Dot] rutgers [Dot] edu}

%\keywords{<keywords>}
\subjclass[2000]{} 

\begin{abstract}  This paper consists of two independent, but related parts.
In the first part we show how to use symbolic computation to derive explicit expressions for the first few moments
of the number of implicants that a random Boolean function has, or equivalently
the number of fixed-dimensional subcubes contained in a random subset of the
$n$-dimensional cube. These explicit expressions suggest, but do not prove,
that these random variables are always asymptotically normal.

The second part presents a full, human-generated proof, of this asymptotic
normality, first proved by Urszula Konieczna.
\end{abstract}

%\dedicatory{}

\maketitle

\section{Introduction}\label{S:intro}

\subsection{Motivation}

Recall that an \emph{implicant} of a Boolean function in $n$ variables, $f(x_1, \dots, x_n)$ is a pure disjunction 
\begin{align}
x_{i_1}^{a_1} \,\wedge \,x_{i_2}^{a_{2}} \wedge \dots \wedge x_{i_r}^{a_r} ,
\end{align}
that implies it. 
Here $1 \leq i_1 < \dots i_r \leq n$, $a_1, \dots, a_r \in \{0,1\}$, $z^1=z$, and $z^0=\bar{z}$ (the negation of $z$).

Fix $r$ and let $n$ vary. We are interested in the statistical distribution of the random variable {\it number of implicants of length $n-r$} in a uniformly-at-random Boolean function
of $n$ variables. Clearly, when $r=0$ it is nothing but the good old (fair) binomial distribution with $2^n$ {\it fair coin-tosses}, $B(\frac{1}{2},2^n)$.

Equivalently, for a random subset of the $n$-dimensional cube, we are interested in the statistical distribution of the number of $r$-dimensional subcubes properly contained in it.

We would like to have \textbf{explicit expressions}, in $n$, for the $k^{th}$ moment of this random variable, for as many as possible $r$ and $k$. This turns out to be
a challenging {\it symbolic-computational} problem that we will address in the first part of this paper.

In the second part we consider asymptotics as \ntoo.
It was proved by  Urszula Konieczna \cite{UK} that 
for each fixed $r$, this distribution is {\it asymptotically normal}.
We will reprove this by a somewhat different method which leads to a simple
proof.

\subsection{Our Random Variables}

The \emph{sample space} is the set of subsets of $\{0,1\}^n$, that has cardinality $2^{2^n}$.

Let's define our random variables formally.
For a (uniformly-at-) random subset $S$, of $\{0,1\}^n$, and \emph{fixed} $r$, define the \textbf{random variable}
\begin{align}
X_r(S):=\text{ number of }r\text{-dimensional subcubes of }S.
\end{align}

For example, if $n=3$ and
\begin{align}
S=\{ 000,001,010,011,100,111\},
\end{align}
we have
\begin{align}
X_0(S)=6 , \quad
X_1(S)=6 , \quad
X_2(S)=1 , \quad
X_3(S)=0 .
\end{align}

We would like to get, for as many pairs $(k,r)$ as possible, {\bf explicit} expressions in $n$, for the $k$-th moment of $X_r$, i.e. for
\begin{align}
f_{kr}(n):=\E[{X_r}^k](n) .
\end{align}

\section{Explicit Expressions for Moments of the Number of Low-Dimensional Subcubes in a Random Subest of the $n$-dimensional cube}\label{S:Expl}

\subsection{The Expectation and Variance}

The {\it first} moment, aka  {\it expectation}, aka {\it mean}, aka {\it average}, is easy, using {\bf linearity of expectation}.

For any specific subcube $C$ of $\{0,1\}^n$, define the {\it atomic} random variable, $X_C$, on subsets, $S$, of $\{0,1\}^n$ as follows.
\begin{align}
X_C(S)\,= \,
\begin{cases}
1,&\text{if} \quad  C \subset S ; \\
0,& \text{otherwise}.
\end{cases}
\end{align}
Let $\cC(n,r)$ be the set of all ${\binom{n}{r}}2^{n-r}$ $r$-dimensional subcubes of $\{0,1\}^n$. Then, of course,
\begin{align}\label{xr}
X_r(S) \, = \, \sum_{C \in \cC(n,r)} X_C(S) .  
\end{align}
Applying the {\it expectation} functional and using the {\bf linearity of expectation}, we get that the average, let's call it 
$\mu_r(n)$, is
\begin{align} \label{LoA}
\mu_r(n) =
\E[X_r] \, = \, \E\lrsqpar{\sum_{C \in \cC(n,r)} X_C} \, = \, \sum_{C \in \cC(n,r)} \E[X_C] .
\end{align}

Now the probability that a random subset of $\{0,1\}^n$ contains an $r$-dimensional subcube $C$ is 
$\left(\frac{1}{2}\right)^{2^r}$, since for each of its vertices, the chance of it belonging to $S$ is $\frac{1}{2}$,
and by {\it independence} the probability that all its $2^r$ vertices belong to $S$ is indeed  $\frac{1}{2^{2^r}}$. The probability that $X_C(S)=0$ is of course   $1-\left(\frac{1}{2}\right)^{2^r}$, hence
\begin{align}\label{mu}  
\E[X_C]= 1 \cdot \left(\frac{1}{2}\right)^{2^r} \, + \, 0 \cdot \left(1-\left(\frac{1}{2}\right)^{2^r}\right) 
\, = \, \left(\frac{1}{2}\right)^{2^r} .
\end{align}
Going back to equation \eqref{LoA} we have
\begin{align}
\mu_r(n) \, = \, \sum_{C \in \cC(n,r)} \E[X_C] = \, \sum_{C \in \cC(n,r)} \frac{1}{2^{2^r}} = |\cC(n,r)| \cdot  \frac{1}{2^{2^r}} = \frac{{\binom{n}{r}}2^{n-r}}{2^{2^r}} .
\end{align}

In a beautiful paper, Thanatipanonda \cite{T} derived an explicit expression for the general second moment, for every $r$-dimensional cube.\\\\

{\bf Thanatipanoda's General Formula for the Second Moment}:
\begin{align}\label{th2}
\E[X_r^2] \, = \,
\sum_{i=0}^{r} \frac{n! 2^{n-i}}{i!(r-i)!^2 (n-2r+i)! 2^{2^{r+1}}} \cdot (2^{2^i} -1) \, 
+ \,\frac{[{\binom{n}{r}} 2^{n-r}]^2}{2^{2^{r+1}}} ,
\end{align}
from which immediately follows, using $[\E(X_r-\mu_r(n))^2]=\E[X_r^2]-\mu_r(n)^2$, the following formula.\\\\

{\bf Thanatipanoda's General Formula for the Variance}:
\begin{align}\label{thvar}
  \text{Var}(X_r)=
\sum_{i=0}^{r} \frac{n! 2^{n-i}}{i!(r-i)!^2 (n-2r+i)! 2^{2^{r+1}}} \cdot
  (2^{2^i} -1)  .
\end{align}
Note that the variance is a polynomial in $(n,2^n)$ of degree $2r$ in $n$ and degree $1$ in $2^n$.

\subsection{Higher Moments}

\subsubsection{Edges}

Thanatipanonda was unable to get such a general formula for higher moments, but did get $\E[X_1^3]$, from which he immediately deduced that
the third central moment (or third-moment-about-the-mean) of $X_1$ is
\begin{align}
\E[(X_1-\mu_1(n))^3] \, = \, \frac{3n^3 2^n}{64} .
\end{align}

Using the symbolic-computational algorithms to be described in the next section, we managed to derive the following explicit formulas
\begin{align}
\E[(X_1-\mu_1(n))^4] =\frac{n 2^{n}}{1024}\big(&12 n^{3} 2^{n} +12 n^{2} 2^{n}+40 n^{3} \notag\\
&+3 n 2^{n}-48 n^{2}+12 n -16\big),
\end{align}

\begin{align}
\E[(X_1-\mu_1(n))^5]  = \frac{5 n^{3} 2^{n} }{1024} \left(6 n^{2}2^{n} +3 n 2^{n}+4 n^{2}-24 n +8\right),
\end{align}

\begin{align}
\E[(X_1-&\mu_1(n))^6] = \notag\\
\frac{n 2^{n}}{32768}\cdot\big(&120 n^{5}\left(2^{n}\right)^{2} +180 n^{4} \left(2^{n}\right)^{2}
+1920 n^{5} 2^{n}+90 n^{3}\left(2^{n}\right)^{2} \notag\\
&-840 n^{4} 2^{n}-1792 n^{5}
+15 n^{2} \left(2^{n}\right)^{2}-360 n^{3} 2^{n} -5280 n^{4}\notag\\
&-300 n^{2} 2^{n}+3840 n^{3}
-240 n 2^{n}+3840 n^{2}-6720 n +4864 \big)
.\end{align}
It follows that the {\bf scaled moments about the mean} for the third, fourth, fifth, and sixth moments, converge, as $n \rightarrow \infty$, to $0,3,0,15$ respectively,  the  respective moments of the normal distribution,
indicating that the random variable $X_1$ (the number of edges contained in $S$)  is {\bf probably} asymptotically normal. To fully prove asymptotic normality, of course, we need to prove
it for all moments, not just for the first six.

\subsubsection{Squares}

We only managed to get explicit expressions for the third and fourth moments for $X_2$.
\begin{align}
\E[(X_2-\mu_2(n))^3] \, = \, 
\frac{2^{n} n \left(n -1\right)}{32768} \left(9 n^{4}+6 n^{3}+21 n^{2}-16 n -34\right),
\end{align}

\begin{align}
\E[(X_2-\mu_2(n))^4&] = \notag\\
\frac{2^{n} n \left(n -1\right)}{4194304} \cdot
\big(&12 n^{6} 2^{n}+12 n^{5} 2^{n} +
520 n^{6}+24 n^{4} 2^{n}-24 n^{5}-12  n^{3}2^{n}+1272 n^{4}
\notag\\
&-9 n^{2}2^{n} -840 n^{3}-27 n 2^{n}-5232 n^{2}-2768 n +240\big)
.\end{align}

\subsubsection{3-dimensional cubes}

We only managed to get an explicit expression for the third moment for $X_3$.
\begin{align}
\E[(X_3-\mu_3(n))^3] =\quad\quad& \notag\\
\frac{2^{n} n \left(n -1\right) \left(n -2\right)}{2415919104} \big(&14 n^{6}+24 n^{5}+479 n^{4}+2046 n^{3}\notag\\
&+6779 n^{2}+15444 n -23112\big)
.\end{align}

\subsection{Our Method}

Obviously we did not derive these formulas by hand. We had to teach our computer how to find them. It also uses {\it linearity of expectation}, but
with higher moments things get very complicated.
Recall that
\begin{align}
X_r(S) \, = \, \sum_{C \in \cC(n,r)} X_C(S) .
\end{align}
Hence, the $k$-th moment is
\begin{align}\label{b1}
  \E\left[(X_r)^k\right]&=
\E\Big[\Big( \sum_{C \in \cC(n,r)} X_C(S) \Big)^k\Big]
\notag\\&
= \sum_{[C_1, \dots, C_k] \in \cC(n,r)^k} \E[X_{C_1} X_{C_2} \cdots X_{C_k}]
.\end{align}

So we sum over all $({\binom{n}{r}}2^{n-r})^k$ members of $\cC(n,r)^k$.  
Since the product $X_{C_1}(S) X_{C_2}(S) \cdots X_{C_k}(S)=1$ if {\bf each} of $C_1,C_2, \dots, C_k$ is properly included in $S$ (i.e. if
each vertex  in $C_1 \cup C_2 \cup \dots C_k$ belongs to $S$), and $0$ otherwise, 
the contribution,
or weight,
due to each such term is
\begin{align}
\text{Weight}([C_1,C_2,\dots,C_k])=
\E\left[X_{C_1} X_{C_2} \cdots X_{C_k}\right]= \frac{1}{2^{|C_1 \cup C_2 \cup \dots C_k|}} .
\end{align}

\subsubsection{Data Structure}

Every $r$-dimensional subcube of $\{0,1\}^n$ has the form
\begin{align}
C=\{ (x_1, \dots, x_n) \in \{0,1\}^n \, | \,x_{i_1}=\alpha_{i_1}, \dots, x_{i_{n-r}}=\alpha_{i_{n-r}} \} ,
\end{align}
for some $1\leq i_1 < i_2 < \dots <i_{n-r} \leq n$ and $(\alpha_{i_1}, \dots \alpha_{i_{n-r}}) \in \{0,1\}^{n-r}$.
A good way to represent it on a computer is as a row-vector of length $n$, in the {\it alphabet} $\{0,1,*\}$,
where the entries corresponding to  $i_1, i_2, \dots, i_{n-r}$  have $\alpha_{i_1}, \dots. \alpha_{i_{n-r}}$ respectively
and the remaining $r$ entries are filled with {\bf wild cards}, denoted by $*$.

For example, if $n=7$ and $r=3$, the $3$-dimensional cube
\begin{align}
\{ (x_1, \dots, x_7) \in \{0,1\}^7 \, | \,x_2=1, x_4=1, x_5=0, x_7=1\} ,
\end{align}
is represented by
\begin{align}
*1*10*1 .
\end{align}

We are trying to find a weighted count of {\bf ordered} $k$-tuples of $r$-dimensional subcubes. The natural data structure for these is the set of $k$ by $n$ matrices in the
`alphabet' $\{0,1,*\}$ where every row has exactly $r$ `wildcards', *.

Let's call this set of matrices, that correspond to $\cC(n,r)^k$, $\cC(n,r,k)$.

For any {\bf specific}, numeric $n$, there are `only'  $(2^{n-r}{\binom{n}{r}})^k$ of these matrices, and for each and every one of them one can find
the cardinality of the union of the corresponding subcubes, let's call it $v$, and add to the running sum $\frac{1}{2^v}$. But we want to do it for
{\bf symbolic} $n$, i.e. for `all' $n$. 
We will soon see how, for each {\it specific} (numeric) $r$ and $k$ this can be done {\it in principle}, but only for relatively small $r$ and $k$ {\it in practice}.
But let's try and push it as far we can.
An interesting consequence of our algorithm is the precise degree in $n$ and $2^n$ of the expression for $\E[X_r^k](n)$.

\subsubsection{The Kernel}

A key object in our approach is the {\bf kernel}. 
Given a $k \times n$ matrix $M$ in the alphabet $\{0,1,*\}$ let's call a column {\bf active} if it contains
at least one `$*$'. Note that the matrix has exactly $k\cdot r$  `$*$'s, hence the number of {\bf active columns}, let's call it $a$,
is between $r$ and $k\cdot r$. \footnote{More generally, if we want to find an expression for the {\it mixed moment} $\E[X_{r_1} \cdots X_{r_k}]$
the number of active columns is between $\max(r_1, \dots, r_k)$ and $r_1 + \dots + r_k$.} The {\bf kernel} of $M$ is
the submatrix of $M$ consisting of its active columns.\\

Let $\cC_a(n,r,k)$ be the subset of $\cC(n,r,k)$ matrices with exactly $a$ active columns. We will call such a matrix in {\bf canonical form} if
the active columns are occupied by the $a$ leftmost columns (i.e. its kernel is contiguous starting in the first column).
Let's denote by $\overline{\cC}_a(n,r,k)$ the set of such matrices in canonical form.
Obviously, there are ${\binom{n}{a}}$ ways to choose which of the $n$ columns are active and hence
\begin{align}
\text{Weight}(\cC(n,r,k)) \, =& \, \sum_{a=r}^{rk} \text{Weight}(\cC_a(n,r,k)) \notag\\
=&\sum_{a=r}^{rk} {\binom{n}{a}} \, \text{Weight}(\overline{\cC}_a(n,r,k)) .
\end{align}
For any set, $S$, $\text{Weight}(S)$ is the sum of the weights of its members. Note that this has degree $rk$ in $n$.

It remains to do a {\it weighted-count}, where every matrix gets `credit' $1/2^v$, where $v$ is the cardinality of the union of the corresponding subcubes represented by the $k$ rows,
for the set $\overline{\cC}_a(n,r,k)$, of matrices in canonical form. Note that there are only {\bf finitely many} choices for the $a$ leftmost columns, i.e. the set of
$k \times a$ matrices in the alphabet $\{0,1,*\}$ with the property that every column has at least one `$*$', and every row has exactly $r$ `$*$'s.
These can be divided into {\it equivalence classes} obtained by permuting rows and columns and transposing $0$ and $1$ in any given column. Once these are sorted into equivalence classes, one needs only
examine one representative, and then multiply the weight by the cardinality of the class.

But what about the $n-a$ rightmost columns? There are $2^{k(n-a)}$ possible submatrices; the alphabet here is $\{0,1\}$.
Almost all of these have distinct rows, more precisely,
\begin{align}
{\binom{2^{n-a}}{k}} k!
\end{align}
of them, and these will produce the smallest possible weight in conjunction with any kernel. The other extreme is that all the rows of the submatrix consisting of the $n-a$ rightmost columns are identical, and then there are only $2^{n-a}$ choices to fill them in.

In general, every such member of $\overline{\cC}_a(n,r,k)$, determines a {\bf set partition} of the set of rows $\{1,\dots,k\}$, where two rows are {\it roomates} if they have the same last $n-a$ entries.
If that set-partition has $m$ members
$1 \leq m \leq k$, then the number of choices of assigning {\bf different} $0-1$ vectors of length $n-a$ to each of the parts of the set-partition is
\begin{align}
{\binom{2^{n-a}}{m}} m! .
\end{align}

Now for each $a$ and for each set-partition, we let the computer generate the {\bf finite} set of $k \times a$ matrices in the alphabet $\{0,1,*\}$. Each of the
members of the set partition has its own submatrix, and we ask our computer to kindly find the number of vertices in the corresponding union of
subcubes corresponding to each member of the examined set partition. Since they are disjoint, we add them up, getting $v$ for that particular pair (matrix, set-partition), giving credit $1/2^v$.

\subsubsection{Implementation}

All this is implemented in the Maple package {\tt SMCboole.txt}, available from:
\\
\href{https://sites.math.rutgers.edu/\~zeilberg/tokhniot/SMCboole.txt}{\tt https://sites.math.rutgers.edu/\~{}zeilberg/tokhniot/SMCboole.txt}

In particular `{\tt Moms(A,n);}', for any list of non-negative integers $A=[r_1, \dots r_k]$ gives you the mixed moment $\E[X_{r_1} \cdots X_{r_k}]$. For example, to get the
third moment of the number of edges (i.e. $1$-dimensional subcubes) type \quad
{\tt Moms([1,1,1],n);}  \quad, which very quickly returns:
\begin{align}
\frac{2^{n} n^{2} \left(\left(2^{n}\right)^{2} n +12 \,2^{n} n +6 \,2^{n}+24 n \right)}{512} .
\end{align}

To get the third moment of the number of squares (i.e. $2$-dimensional subcubes), type \quad {\tt Moms([2,2,2],n);} \quad,
getting
\begin{align}
\frac{2^{n} n  \left(n -1\right)}{2097152} \cdot\bigg(&\left( 2^{n}\right)^{2} n^{4}-2 \left(2^{n}\right)^{2} n^{3}
+48 \,2^{n} n^{4}+\left(2^{n}\right)^{2} n^{2} \notag\\
&+576 n^{4}+24 \,2^{n} n^{2}+384 n^{3}-72 \,2^{n} n \notag\\
&+1344 n^{2}-1024 n -2176\bigg) .
\end{align}

The third moment of the number of $3$-dimensional cubes takes a bit longer, and we were unable to compute the fourth moment of the number of $3$-dimensional cubes; it took too much time and too much space.

More informative for {\it statistical purposes} are the {\bf central moments}, $\E[(X_r-\mu_r(n))^k]$, that  Maple easily derives,
using {\it linearity of expectation} from the pure moments. The function call for this is
{\tt MOMrk(r,k,n);} where {\tt r} and {\tt k} are {\it numeric} but $n$ is a {\bf symbol} denoting the dimension of the ambient cube.

To get the explicit expression given above 
for the third through sixth moments for the number of edges, the third and fourth moments for the number of squares, and the third moment for the number of $3$-dimensional
subcube (all {\it about the mean}) we typed:

{\tt MOMrk(1,3,n);} ,  {\tt MOMrk(1,4,n);} ,  {\tt MOMrk(1,5,n);} ,  {\tt MOMrk(1,6,n);} , 

{\tt MOMrk(2,3,n);} ,  {\tt MOMrk(2,4,n);} ,  

{\tt MOMrk(3,3,n);} 

respectively. To our chagrin, `{\tt MOMrk(3,4,n);}'   took too long.

\subsubsection{Consequence of the algorithm}: The $k$-th  moment of $X_r$ is a bivariate polynomial in $(n,2^n)$ of degree $k\,r$ in $n$ and degree $k$ in $2^n$ .

This raises the {\it theoretical} possibility (in God's computer) of finding these expressions by {\bf pure brute force}. The generic polynomial in $(n,2^n)$ of degree $kr$ in $n$ and degree $k$ in $2^n$
has $(1+kr)(1+k)$ `degrees of freedom'. So using {\it undetermined coefficients} we need to generate a table of $\E[X_r^k](n)$ for $1 \leq n \leq (1+kr) \cdot (1+k)$. 
After gathering the data, we use linear algebra to solve a system of  $(1+kr) \cdot (1+k)$ equations with that many unknowns.
For each specific
$n=n_1$ there are `only' $2^{2^{n_1}}$ subsets, and for each of them we can ask how many $r$-dimensional subcubes do they contain, raise it to the $k$-th power and take the average.
Alas, $2^{2^5}$ is already big enough, so only God's computer, with practically infinite time and space,  can carry this brute force approach.

\section{Asymptotic Normality }\label{S:ANor}

As said in the introduction,
asymptotic normality of $X_{r}$ was proved by  Urszula Konieczna \cite{UK}. 
The random variable $X_r$ depends on $n$, and for emphasis, we denote it in
the sequel by $X_{n,r}$.
We may then state the result as follows,
where $N(0,1)$ denotes
the standard normal distribution.

\begin{theorem}[\cite{UK}]\label{T1}
For every fixed $r\ge0$, %the normalized random variables 
we have as \ntoo
\begin{align}\label{t1}
\frac{X_{n,r}-\mu_r(n)}{\sqrt{\Var\xpar{X_{n,r}}}}
\dto N(0,1)  
\end{align}
with convergence in distribution
and of all moments.  
\end{theorem}

As already noted at the very beginning of this paper, the case $r=0$ is
well known, since then the variables $X_C$ are 
independent and this becomes an instance of the central limit theorem for
binomial variables, known since a pamphlet by de Moivre in 1733.
(This case can be used as a sanity check in the arguments below.)

The proof in \cite{UK} is based on the method of moments.
We will in this section survey how a variation of this method, based on
{\bf cumulants} and {\bf dependency graphs} leads to a quick proof with
almost no calculations.

\begin{remark}
Urszula Konieczna \cite{UK} 
considered a more general case where each subcube appears with
probability $p\in(0,1)$, where $p=p(n)$ may depend on $n$, and showed
asymptotic normality for a large range of $p(n)$.
A special case, not including the case $p=\frac12$ treated here, was earlier
shown by Karl Weber \cite{KW}.
This extension can also be treated by the method below, but for simplicity
we continue to consider only the case $p=\frac12$ discussed in \refS{S:Expl}.
\end{remark}

The aim is thus to show
that the leading terms in the expressions for the
central moments 
$\E[(X_{n,r}-\mu_r(n))^k]$ are what we would expect from a normal distribution, or,
more formally, that the scaled central moments
\begin{align}\label{a1}
\E\lrsqpar{\lrpar{\frac{X_{n,r}-\mu_r(n)}{\sqrt{\Var\xpar{X_{n,r}}}}}^k}
\end{align}
converge as \ntoo{}
to the corresponding moment $\E [Z^k]$ of a standard normal random
variable $Z$, for every fixed $r\ge0$ and $k\ge1$.
Then, the random variables $X_{n,r}$ are asymptotically normal by the method of
moments.

We thus want to find
the leading term of the moment \eqref{a1}. In \refS{S:Expl}, we computed the
central moments from the pure moments. This involves massive cancellation of
high-order terms, and is less suitable for a (human) proof for general $k$.

As a first step, we may modify \eqref{b1} and instead expand the central
moment as
\begin{align}\label{b1b}
  \E\bigsqpar{(X_{n,r}-\mu_r(n))^k}&
%=\E\left[\left( \sum_{C \in \cC(n,r)} X_C(S) \right)^k\right]
%\notag\\&
= \sum_{[C_1, \dots, C_k] \in \cC(n,r)^k} \E\bigsqpar{(X_{C_1}-\mu) \cdots (X_{C_k}-\mu)}
,\end{align}
where $\mu=\E[X_C]$ is given by \eqref{mu}.
Note that two variables $X_{C_1}$ and $X_{C_2}$ are independent if
$C_1\cap C_2=\emptyset$. In particular, for $k=2$, all such terms in the sum
in \eqref{b1b} vanish; this leaves only 
$\Theta(n^{2r}2^n)$ terms of the $\Theta\bigpar{n^{2r}2^{2n}}$ terms in the
full sum, or in \eqref{b1}, which is reflected in the leading terms in
\eqref{th2} and \eqref{thvar}.

For higher moments, we can obtain a further reduction by considering 
{\bf cumulants} (also called {\bf semiinvariants}) instead of moments.
We describe this briefly, referring to e.g.\
\cite{LeonovShiryaev} or \cite[pp.~145--149]{JLR} for details.
In general, if $Y$ is a  random variable with moments $m_k:=\E [Y^k]$ and
moment generating function $\sum_{k=0}^\infty m_k t^k/k!$, then the cumulants
$\kk_k=\kk_k(Y)$ are the coefficients of the generating function
\begin{align}\label{kk}
  \sum_{k=1}^\infty\kk_k \frac{t^k}{k!}
:=
\log \lrpar{  \sum_{k=0}^\infty m_k \frac{t^k}{k!}}.
\end{align}
(The generating functions can be regarded as  formal power series.)
Moreover, this generalizes to mixed cumulants $\kk_k(Y_1,\dots,Y_k)$ of
several random variables.
This means that there are algebraic relations: each cumulant
$\kk_k$ is a polynomial
in the moments $m_\ell$ of order $\ell\le k$, and conversely, and this
generalizes to mixed cumulants and moments.
For example, $\kk_1=m_1=\E Y$, and $\kk_2=m_2-m_1^2=\Var(Y)$, and the mixed
cumulant $\kk(Y_1,Y_2)=\E[Y_1Y_2]-\E[Y_1]\E[Y_2]$, the covariance of $Y_1$
and $Y_2$.

A convenient property of cumulants is that a normal distribution
$N(\mu,\gss)$ has moment generating function 
$\exp\bigpar{\mu t + \tfrac{\gss}2t^2}$, and thus by \eqref{kk} cumulants
$\kk_k=0$ for $k\ge3$. Consequently, to show that a normalized sequence of
random variables is asymptotically normal, it suffice to show that each
cumulant $\kk_k$ with $k\ge3$ converges to 0.

The mixed cumulants are multilinear, and thus we have an analog of
\eqref{b1} and \eqref{b1b}:
\begin{align}\label{kkk}
\kk_k\bigpar{X_{n,r}}&
=\kk\bigpar{X_{n,r},\dots,X_{n,r}}&
= \sum_{[C_1, \dots, C_k] \in \cC(n,r)^k} \kk\bigpar{X_{C_1}, \dots, X_{C_k}}
.\end{align}
We can here reduce the number of terms in the sum by a general
property of cumulants: If there is a partition $\set{1,\dots,k}=I\cup J$
into two non-empty sets $I$ and $J$ such the two families of random
variables
$(Y_i\mid i\in I)$ and $(Y_j\mid j\in J)$ are independent of each other,
then the mixed cumulant $\kk(Y_1,\dots,Y_k)=0$.
(Informally, this means that the $k$th cumulant $\kk_k(X_{n,r})$ can be seen
as some kind of measure of $k$th order dependencies in the sum \eqref{xr},
with lower order dependencies removed.)

As noticed above, each of our variables $X_C$ is independent of most of the
others. We need to keep track of not just pairwise independence, but also
independence between families of such variables; it is then convenient to
use the notion of {\bf dependency graphs}.

\begin{definition}
  Let $(Y_\ga)_{\ga\in\cA}$ be a family of random variables, with some
  arbitrary index set $\cA$.
A  {\bf dependency graph} for this family is a graph $\gG$ with vertex set
$\cA$, such that if $I$ and $J$ are two disjoint subsets of $\cA$ such that
there is no edge in $\gG$ with one endpoint in $I$ and the other in $J$,
then the 
two families $(Y_\ga\mid \ga\in I)$ and $(Y_\gb\mid \gb\in J)$ 
are independent of each other.
\end{definition}

\begin{lemma}\label{LgG}
  Let $\gG(n,r)$ be the graph with vertex set $\cC(n,r)$ and an edge between
  two distinct cubes $C,C'\in\cC(n,r)$ iff\/ $C\cap C'\neq\emptyset$.
Then $\gG(n,r)$ is a dependency graph for the family $\xpar{X_C\mid C\in\cC(n,r)}$.
\end{lemma}
\begin{proof}
Let $I$ and $J$ be disjoint subsets of $\cC(n,r)$ with no edge between $I$
and $J$.
Then $V_I:=\bigcup_{C\in I}C$  and $V_J:=\bigcup_{C\in J}C$ are two disjoint
subsets of $\setoi^n$. Thus the two random sets $S\cap V_I$ and $S\cap V_J$
are independent. Since the random variables $\xpar{X_C\mid C\in I}$ only
depend on $S\cap V_I$, and similarly for $J$, 
it follows that the families
$\xpar{X_C\mid C\in I}$ and $\xpar{X_C\mid C\in J}$ are independent.
\end{proof}

We can now use the dependency graphs $\gG(n,r)$ and the machinery above to
estimate cumulants of $X_{n,r}$ and show that
$\kk_k\bigpar{X_{n,r}}=o\bigpar{(\Var X_{n,r})^{k/2}}$ as \ntoo{} for every
fixed $r\ge0$ and $k\ge3$. However, we do not have to do these calculations,
since they already have been done under general hypotheses, leading to
convergence theorems that are easy to apply in situations like ours where
there is a rather sparse dependency graph.
We use here the following theorem, 
taken (with minor changes in notation) from %\cite{SJ58}
\cite[Theorem 2, with moment convergence by its proof and Theorem 1]{SJ58}.
See \eg{} \cite[Theorems 6.18 and 6.20]{JLR} for some related theorems 
(proved by the same cumulant method) that also can be used to show \refT{T1},
and \cite[Theorem~6.33]{JLR} for another related theorem (proved by a
different method) that also yields the asymptotic normality in \refT{T1}
(but not immediately moment convergence).

\begin{theorem}[\cite{SJ58}]\label{T58}
  Suppose that, for each $n$,
$\xpar{Y_{n,\ga}\mid \ga\in\cA_n}$ is a family of bounded random variables:
$|Y_{n,\ga}|\le A_n$. Let $N_n:=|\cA_n|$.
Suppose further that $\gG_n$ is a dependency graph for this family, and
let $M_n$ be the maximal degree of $\gG_n$ (assuming this is not $0$, in which
case we let $M_n:=1$). Let $\SW_n=\sum_{i\in\cA_i} Y_{n,i}$ and
$\gss_n:=\Var(\SW_n)$.
If there exists a positive integer $m$ such that
\begin{align}
  \label{t58a}
(N_n/M_n)^{1/m} M_n A_n/\gs_n \to0
\qquad\text{as \ntoo},
\end{align}
then
\begin{align}
\bigpar{\SW_n-\E \SW_n}/\gs_n\dto N(0,1)
  \qquad\text{as \ntoo}
,\end{align}
with convergence in distribution and with all moments.
\end{theorem}

\begin{proof}[Proof of \refT{T1}]
Fix $r\ge0$. We let $c_1, c_2,c_3$ denote 
unimportant positive constants that may
depend on $r$, but not on $n$.

We use \refT{T58}, with 
$Y_{n,C}:=X_C$ for
$C\in\cA_n:=\cC(n,r)$, so $\SW_n=X_{n,r}$.
We use \refL{LgG} and take the dependency graph $\gG_n=\gG(n,r)$.
Then
\begin{align}\label{swN}
  N_n=|\cC(n,r)|=\binom{n}{r}2^{n-r} \le n^r 2^n,
\end{align}
and, by \eqref{thvar} (where $i=0$ yields the leading term),
\begin{align}
  \gss_n:=\Var(X_{n,r})\ge \cc n^{2r} 2^n.
\end{align}
Furthermore, we may take 
\begin{align}
  A_n:=1.
\end{align}
Moreover, from the definition of the graph $\gG(n,r)$ follows that it is
regular, with degree
\begin{align}\label{swM}
  M_n \le 2^r\binom{n}{r}\le\cc n^r.
\end{align}
We choose $m=3$ in \refT{T58}. (This is the standard case, which corresponds
to the method of moment without extra arguments, see the proof in
\cite{SJ58}.)
Then \eqref{swN}--\eqref{swM} yield
\begin{align}
(N_n/M_n)^{1/m} M_n A_n/\gs_n 
=N_n^{1/3}M_n^{2/3}A_n/\gs_n
\le \cc\frac{ n^r 2^{n/3}}{n^r 2^{n/2}}
\to0
\end{align}
as \ntoo, so \eqref{t58a} holds. Consequently, \refT{T58} applies and shows
\refT{T1}.
\end{proof}

\section{Conclusion }\label{S:Conc}

In the first part, we demonstrated  the power of computers to {\it automatically} derive complicated explicit expressions, all polynomials in $n$ and $2^n$, for low moments of the
random variable {\it number of fixed-size implicants} of a Boolean function of $n$ variable, or equivalently, the number of fixed-dimensional subcubes of
a random subset of the $n$-dimensional unit cube. These implied that the scaled first few moments for low-dimensional subcubes
tend, as $n$ goes to infinity, to those of the Normal Distribution, suggesting, but by no means proving, that for each fixed dimension, this random variable 
is {\it asymptotically normal}.

In the second part, we demonstrated the power of {\it human ingenuity}, and
the traditional lore of probability and statistics, by presenting a simple proof
of asymptotic normality, first proved in \cite{UK}.

\def\nobibitem#1\par{}

\end{document}